\documentclass[smallextended]{svjour3}       
\smartqed  
\usepackage{amsfonts}
\usepackage{amsmath}
\usepackage{graphicx}
\usepackage[english]{babel}

\newcommand{\Sum}{\displaystyle\sum}

\begin{document}

\title{A generalization of Bohr's Equivalence Theorem}


\author{J.M. Sepulcre \and
        T. Vidal
} \institute{J.M. Sepulcre  \and T. Vidal   \at
              Department of Mathematics\\
              University of Alicante\\
              03080-Alicante, Spain\\
             \email{JM.Sepulcre@ua.es, tmvg@alu.ua.es}
}

\date{Received: date / Accepted: date}
\maketitle
\begin{abstract}
Based on a generalization of Bohr's equivalence relation for general Dirichlet series, in this paper we
study the sets of values taken by certain classes of equivalent almost periodic functions in their strips of almost periodicity. In fact, the main result of this paper consists of a result like Bohr's equivalence theorem  extended to the case of these functions.
\end{abstract}

\section{Introduction}

The main motivation of this paper arises from the work of the Danish mathematician Harald Bohr concerning both the equivalence relation for general Dirichlet series, which is named after him, and almost periodic functions, whose theory was created and developed in its main features by himself.

On the one hand, general Dirichlet series consist of those exponential sums that take the form
$$\Sum_{n\geq 1}a_ne^{-\lambda_n s},\ a_n\in\mathbb{C},\ s=\sigma+it,$$
 where
$\{\lambda_n\}$ is a strictly increasing sequence of posi\-tive numbers
tending to infinity. 
In particular, it is widely known that the Riemann zeta function $\zeta(s)$, which plays a pivotal role
in analytic number theory, is defined as the analytic continuation of the function defined for $\sigma>1$ by the sum
$\sum_{n=1}^{\infty}\frac{1}{n^s}$, which constitutes a classical Dirichlet series.

In the beginnings of the 20th century, H. Bohr gave important steps in the understanding
of Dirichlet series and their regions of convergence, uniform convergence and absolute
convergence. As a result of his investigations on these functions, he introduced an equivalence relation among them that led to so-called Bohr's equivalence theorem, which shows that equivalent Dirichlet series take the same values in certain vertical lines or strips in the complex plane (e.g. see \cite{Apostol,BohrDirichlet,Rigue,Spira}). 

On the other hand, Bohr also developed during the $1920$'s the theory of almost periodic functions, which opened a way to study a wide class of trigonometric series of the general type and even exponential series (see for example \cite{Besi,Bohr,Bohr2,Corduneanu1,Jessen}). The space of almost periodic functions in a vertical strip $U\subset \mathbb{C}$, which will be denoted in this paper as $AP(U,\mathbb{C})$, coincides with the set of the functions which can be approximated uniformly in every reduced strip 
of $U$ by exponential polynomials $a_1e^{\lambda_1s}+a_2e^{\lambda_2s}+\ldots+a_ne^{\lambda_ns}$ with complex coefficients $a_j$ and real exponents $\lambda_j$ (see for example \cite[Theorem 3.18]{Corduneanu1}). 
Moreover, S. Bochner observed that the almost periodicity of a function $f$ in a vertical strip $U$ is equivalent to that every sequence $\{f(s + it_n )\},\ t_n\in\mathbb{R},$ of vertical translations of $f$ has a subsequence that converges uniformly for $s$ in $U$.

The purpose of this paper is to try to extend Bohr's equivalence theorem, which concerns to Dirichlet series, to more general classes of almost periodic functions in $AP(U,\mathbb{C})$.
In this respect note that the exponential polynomials and the general Dirichlet series are a particular family of exponential sums or, in other words, expressions of the type
$$P_1(p)e^{\lambda_1p}+\ldots+P_j(p)e^{\lambda_jp}+\ldots,$$
where the $\lambda_j$'s are complex numbers and the $P_j(p)$'s are polynomials in the parameter $p$.
Precisely, by analogy with Bohr's theory, we established in \cite{SV} an equi\-valence relation $\sim$ on the classes $\mathcal{S}_{\Lambda}$ consisting of exponential sums of the form
 \begin{equation}\label{eqqnew}
\sum_{j\geq 1}a_je^{\lambda_jp},\ a_j\in\mathbb{C},\ \lambda_j\in\Lambda,
\end{equation}
where $\Lambda=\{\lambda_1,\lambda_2,\ldots,\lambda_j,\ldots\}$ is an arbitrary countable set of distinct real numbers (not necessarily unbounded), which are called a set of exponents or frequencies. In this paper, we will suppose that $\Lambda$ is a set of exponents for which there exists an integral basis, which means that each exponent $\lambda_j$ is expressible as a finite linear combination, with entire components, of terms of the basis.


From this equivalence relation $\sim$, we will show that every equivalence class in $AP(U,\mathbb{C})/\sim$, associated with a set of exponents which has an integral basis, 
is connected with a certain auxiliary function that originates the sets of values taken by this equivalence class
along a given vertical line included in the strip of almost periodicity (see propositions \ref{pult} and \ref{pultnew} in this paper). This leads us to formulate and prove Theorem \ref{beqg} in this paper, which constitutes the announced extension of Bohr's equivalence theorem. 
Specifically, we prove that two equivalent almost periodic functions, whose associated Dirichlet series have the same set of exponents for which there exists an integral basis, take the same values on any open vertical strip included in their strip of almost periodicity $U$. Moreover, Example \ref{exi} shows that, if we fix an open vertical strip in $U$, the fact that both almost periodic functions take the same values on it does not imply the equivalence of the two functions. 

\section{Equivalence of exponential sums and almost periodic functions}

We first recall the equivalence relation, based on that of \cite[p.173]{Apostol} for general Dirichlet series, which was defined in \cite{SV} in a more general context. 

\begin{definition}\label{DefEquiv}
Let $\Lambda$ be an arbitrary countable subset of distinct real numbers, $V$ the $\mathbb{Q}$-vector space generated by $\Lambda$ ($V\subset \mathbb{R}$), and $\mathcal{F}$
the $\mathbb{C}$-vector space of arbitrary functions $\Lambda\to\mathbb{C}$. 
We define
a relation $\sim$ on $\mathcal{F}$ by $a\sim b$ if there exists a $\mathbb{Q}$-linear map $\psi:V\to\mathbb{R}$ such that
$$b(\lambda)=a(\lambda)e^{i\psi(\lambda)},\ (\lambda\in\Lambda).$$
\end{definition}


Let
$G_{\Lambda}=\{g_1, g_2,\ldots, g_k,\ldots\}$ be a basis of the
vector space over the rationals generated by a set $\Lambda=\{\lambda_1,\lambda_2,\ldots,\lambda_j,\ldots\}$ of exponents, 
which implies that $G_{\Lambda}$ is linearly independent over the rational numbers and each $\lambda_j$ is expressible as a finite linear combination of terms of $G_{\Lambda}$, say
\begin{equation}\label{errej}
\lambda_j=\sum_{k=1}^{q_j}r_{j,k}g_k,\ \mbox{for some }r_{j,k}\in\mathbb{Q}.
\end{equation}
By abuse of notation, we will say that $G_{\Lambda}$ is a basis for $\Lambda$. Moreover, we will say that $G_{\Lambda}$ is an integral basis for $\Lambda$ when $r_{j,k}\in\mathbb{Z}$ for any $j,k$.

The equivalence relation above can be naturally extended to the classes $S_\Lambda$ of exponential sums of type \eqref{eqqnew}.

\begin{definition}\label{DefEquiv00}
Given $\Lambda=\{\lambda_1,\lambda_2,\ldots,\lambda_j,\ldots\}$ a set of exponents, consider $A_1(p)$ and $A_2(p)$ two exponential sums in the class $\mathcal{S}_{\Lambda}$, say
$A_1(p)=\sum_{j\geq1}a_je^{\lambda_jp}$ and $A_2(p)=\sum_{j\geq1}b_je^{\lambda_jp}.$
We will say that $A_1$ is equivalent to $A_2$ if $a\sim b$, where $a,b:\Lambda\to\mathbb{C}$ are the functions given by $a(\lambda_j):=a_j$ y $b(\lambda_j):=b_j$, $j=1,2,\ldots$ and $\sim$ is in Definition \ref{DefEquiv}.
\end{definition}

Consider $A_1(p)$ and $A_2(p)$ two exponential sums in the class $\mathcal{S}_{\Lambda}$, say
$A_1(p)=\sum_{j\geq1}a_je^{\lambda_jp}$ and $A_2(p)=\sum_{j\geq1}b_je^{\lambda_jp}.$
Fixed an integral basis $G_{\Lambda}$ for $\Lambda$, for each $j\geq1$ let $\mathbf{r}_j$ be the vector of integer components satisfying (\ref{errej}).  
Then the equivalence of $A_1$ and $A_2$ 
can be characterized from the existence of $\mathbf{x}_0=(x_{0,1},x_{0,2},\ldots,x_{0,k},\ldots)\in \mathbb{R}^{\sharp G_{\Lambda}}$
such that $b_j=a_j e^{<\mathbf{r}_j,\mathbf{x}_0>i}$ for every $j\geq 1$ (see \cite[Proposition 1]{SV}). 
In fact, all the results of \cite{SV}
which can me formulated in terms of an integral basis are also valid under Definition \ref{DefEquiv00}.




From the result above, it is clear that the set of all exponential sums $A(p)$ in an equivalence class $\mathcal{G}$ in $\mathcal{S}_{\Lambda}/\sim$, where $\Lambda$ has an integral basis, can be determined by a function $E_{\mathcal{G}}:\mathbb{R}^{\sharp G_{\Lambda}}\rightarrow \mathcal{S}_{\Lambda}$ of the form
\begin{equation}\label{2.4.000}
E_{\mathcal{G}}(\mathbf{x}):=\sum_{j\geq1}a_je^{<\mathbf{r}_j,\mathbf{x}>i}e^{\lambda_jp}\text{, }
\mathbf{x}=(x_1,x_2,\ldots,x_k,\ldots)\in%
\mathbb{R}^{\sharp G_{\Lambda}},
\end{equation}%
where $a_1,a_2,\ldots,a_j,\ldots$ are the coefficients of an exponential sum in $\mathcal{G}$ and the $\mathbf{r}_j$'s are the vectors of integer components associated with an integral basis $G_{\Lambda}$ for $\Lambda$. 

In particular, in this paper we are going to use Definition \ref{DefEquiv00} for the case of exponential sums in $\mathcal{S}_{\Lambda}$ of a complex variable $s=\sigma+it$. 
Precisely,
when the formal series in $\mathcal{S}_{\Lambda}$ are handled as exponential sums of a complex variable on which we fix a summation procedure, from equivalence class generating expression (\ref{2.4.000}) we can consider an auxiliary function as follows.

\begin{definition}\label{auxuliaryfunc0}
Given $\Lambda=\{\lambda_1,\lambda_2,\ldots,\lambda_j,\ldots\}$ a set of exponents which has an integral basis, let $\mathcal{G}$ be an equivalence class in $\mathcal{S}_{\Lambda}/\sim$ and $a_1,a_2,\ldots,a_j,\ldots$ be the coefficients of an exponential sum in $\mathcal{G}$.
For each $j\geq 1$ let $\mathbf{r}_j$ be the vector of integer components satis\-fying the equality $\lambda_j=<\mathbf{r}_j,\mathbf{g}>=\sum_{k=1}^{q_j}r_{j,k}g_k$, where
$\mathbf{g}:=(g_1,\ldots,g_k,\ldots)$ is the vector of the elements of an integral basis $G_{\Lambda}$ for $\Lambda$. Suppose that some elements in $\mathcal{G}$, handled as exponential sums of a complex variable $s=\sigma+it$, are summable on at least  a certain set $P$ included in the real axis by some prefixed summation method.
Then we define the auxiliary function  $F_{\mathcal{G}}: P
\times
\mathbb{R}
^{\sharp G_{\Lambda}}\rightarrow
\mathbb{C}$
 associated with $\mathcal{G}$, relative to the basis $G_{\Lambda}$, as
\begin{equation}\label{2.4.000p}
F_{\mathcal{G}}(\sigma,\mathbf{x}):=\sum_{j\geq1}a_je^{<\mathbf{r}_j,\mathbf{x}>i}e^{\lambda_j\sigma}\text{, }
\sigma\in P,\ \mathbf{x}=(x_1,x_2,\ldots,x_k,\ldots)\in%
\mathbb{R}^{\sharp G_{\Lambda}},
\end{equation}%
where the series in (\ref{2.4.000p}) is summed by the prefixed summation method, applied at $t=0$ to the exponential sum obtained from the generating expression (\ref{2.4.000}) with $p=\sigma+it$.
\end{definition}

In particular, see Definition \ref{auxuliaryfunc} which concerns the case of Bochner-Fej\'{e}r summation method and the set $P$ above is formed by the real projection of the strip of almost periodicity of the corresponding exponential sums.

In other words, the auxiliary function $F_{\mathcal{G}}(\sigma,\mathbf{x})$ can be viewed as the composition $F_{\mathcal{G}}=M\circ E_{\mathcal{G}}$, of the class generating expression considered as $E_{\mathcal{G}}:(\sigma,\mathbf{x})\to \sum_{j\geq1}a_je^{<\mathbf{r}_j,\mathbf{x}>i}e^{\lambda_j(\sigma+it)}\in \mathcal{S}_{\Lambda}$ with the application $M:\mathcal{S}_{\Lambda}\rightarrow \mathbb{C}$ which to an exponential sum $A(t)$ generated by (\ref{2.4.000}) assigns a complex number obtained as the summation of $A(t)$ at $t=0$ by the prefixed summation method.

On the other hand, the space of almost periodic functions $AP(U,\mathbb{C})$ coincides with the set of the functions which can be
approximated uniformly in every reduced strip of $U$ by exponential polynomials with complex coefficients
and real exponents (see \cite[Theorem 3.18]{Corduneanu1}). These aproximating  finite exponential sums can be found by
Bochner-Fej\'{e}r's summation (see, in this regard, \cite[Chapter 1, Section 9]{Besi}). Moreover, each sequence of exponential polynomials that converges uniformly to a function $f\in AP(U,\mathbb{C})$ also converges formally to an unique exponential sum, which is called the Dirichlet series of $f$. In this context, we will see in this paper the strong link
between the sets of values in the complex plane taken by such a function, its Dirichlet series and its associated
auxiliary function.



\begin{definition}\label{DF}
Let $\Lambda=\{\lambda_1,\lambda_2,\ldots,\lambda_j,\ldots\}$ be an arbitrary countable set of distinct real numbers. We will say that a function $f:U\subset\mathbb{C}\to\mathbb{C}$ 
is in the class $\mathcal{D}_{\Lambda}$ if it is an almost periodic function in $AP(U,\mathbb{C})$ 
whose associated Dirichlet series 
is of the form
 \begin{equation}\label{eqqo}
\sum_{j\geq 1}a_je^{\lambda_js},\ a_j\in\mathbb{C},\ \lambda_j\in\Lambda,
\end{equation}
where $U$ is a strip of the type $\{s\in\mathbb{C}: \alpha<\operatorname{Re}s<\beta\}$, with $-\infty\leq\alpha<\beta\leq\infty$.
\end{definition}

Any almost periodic function in $AP(U,\mathbb{C})$ 
is determined by its Dirichlet series, 
which is of type (\ref{eqqo}). 
In fact it is convenient to remark that, even in the case that the sequence of the partial sums of its Dirichlet series 
does not converge uniformly, there exists a sequence of finite exponential sums, the Bochner-Fej\'{e}r polynomials, of the type $P_k(s)=\sum_{j\geq 1}p_{j,k}a_je^{\lambda_js}$ 
where for each $k$ only a finite number of the factors $p_{j,k}$ differ from zero, which converges uniformly to $f$ in every reduced strip in $U$ 
and converges formally to the Dirichlet series  
\cite[Polynomial approximation theorem, pgs. 50,148]{Besi}.

Moreover, the equivalence relation of Definition \ref{DefEquiv00} can be immediately adapted to the case of the functions (or classes of functions) which are identifiable by their also called Dirichlet series, in particular to the classes $\mathcal{D}_{\Lambda}$. 
More specifically, see Definition 5 of \cite[Section 4]{SV} referred to the Besicovitch space which contains the classes of functions which are associated with Fourier or Dirichlet series and for which the extension of our equivalence relation makes sense. For more information on the Besicovitch space, see \cite[Section 3.4]{Corduneanu}.

\section{The auxiliary functions associated with the classes $\mathcal{D}_{\Lambda}$}\label{sv}

Based on Definition \ref{auxuliaryfunc0}, applied to our particular case of almost periodic functions with the Bochner-Fej\'{e}r summation method, note that to every almost periodic function $f\in \mathcal{D}_{\Lambda}$, with $\Lambda$ a set of exponents which has an integral basis, we can associate an auxiliary function $F_f$ of countably many real variables as follows.

\begin{definition}\label{auxuliaryfunc}
Given $\Lambda=\{\lambda_1,\lambda_2,\ldots,\lambda_j,\ldots\}$ a set of exponents which has an integral basis, let $f(s)\in\mathcal{D}_{\Lambda}$ be an almost periodic function in $\{s\in\mathbb{C}:\alpha<\operatorname{Re}s<\beta\}$, $-\infty\leq\alpha<\beta\leq\infty$, whose Dirichlet series is given by $\sum_{j\geq 1}a_je^{\lambda_js}$.
For each $j\geq 1$ let $\mathbf{r}_j$ be the vector of integer components satisfying the equality $\lambda_j=<\mathbf{r}_j,\mathbf{g}>=\sum_{k=1}^{q_j}r_{j,k}g_k$, where
$\mathbf{g}:=(g_1,\ldots,g_k,\ldots)$ is the vector of the elements of an integral basis $G_{\Lambda}$ for $\Lambda$.
Then we define the auxiliary function  $F_f: (\alpha,\beta)
\times
\mathbb{R}
^{\sharp G_{\Lambda}}\rightarrow
\mathbb{C}$
 associated with $f$, relative to the basis $G_{\Lambda}$, as
\begin{equation}\label{2.4}
F_{f}(\sigma,\mathbf{x}):=\sum_{j\geq1}a_j e^{\lambda_j\sigma
}e^{<\mathbf{r}_j,\mathbf{x}>i}\text{, }\sigma \in
(\alpha,\beta)
\text{, }\mathbf{x}=(x_1,x_2,\ldots,x_k,\ldots)\in%
\mathbb{R}^{\sharp G_{\Lambda}},
\end{equation}%
where series (\ref{2.4}) is summed by Bochner-Fej\'{e}r procedure, applied at $t=0$ to the exponential sum $\sum_{j\geq1}a_j e^{<\mathbf{r}_j,\mathbf{x}>i}e^{\lambda_js}$.
\end{definition}

If $f\in AP(U,\mathbb{C})$, whose set of associated exponents $\{\lambda_1,\lambda_2,\ldots\}$ has an integral basis, it was proved in \cite[Lemma 3]{SV} that any function of its equivalence class is also included in $AP(U,\mathbb{C})$. Then we first note that, if $\sum_{j\geq1}a_j e^{\lambda_js}$ is the Dirichlet series of $f\in AP(U,\mathbb{C})$, for every choice of $\mathbf{x}\in\mathbb{R}^{\sharp G_{\Lambda}}$, the sum $\sum_{j\geq1}a_j e^{<\mathbf{r}_j,\mathbf{x}>i}e^{\lambda_js}$ represents the Dirichlet series of an almost periodic function. 

We second note that if the Dirichlet series of $f(s)\in AP(U,\mathbb{C})$ converges uniformly on $U=\{s\in\mathbb{C}:\alpha<\operatorname{Re}s<\beta\}$, then $f(s)$ coincides with its Dirichlet series and (\ref{2.4}) can be viewed as summation by partial sums or ordinary summation.

For the case of the partial sums of the Riemann zeta function $\zeta_n$, the auxiliary function $F_{\zeta_n}$ is called in \cite[p. 163]{Spira2} the ``companion function"\ of $\zeta_n$ (see also \cite[Theorem 3.1]{Avellar} or \cite[Theorem 1]{nonisolation} for the case of exponential polynomials).


In addition, we third note that the Dirichlet series $\sum_{j\geq 1}a_je^{\lambda_js}$, associated with a function $f\in\mathcal{D}_{\Lambda}$ such that $\Lambda$ has an integral basis, arises from its auxiliary function $F_f$ by a special choice of its variables, that is $F_f(\sigma,t\mathbf{g})=\sum_{j\geq 1}a_je^{\lambda_j(\sigma+it)}$. In fact, as we will see in this section, there is a strong link between the sets of values in the complex plane taken by both functions.

In this respect, under the assumption that the set of exponents or frequencies is finite, or it has an integral basis, it is clear that the auxiliary function $F_f(\sigma,\mathbf{x})$ is periodic in each of its coordinates $x_k$, $k\geq 1$, and $\sharp G_{\Lambda}$ is related to the dimension of the higher dimensional space in which the given function, for a fixed value of $\sigma$, can be embedded as a periodic function in each of its coordinates.

We next show a characterization of the property of equivalence of functions in the classes $\mathcal{D}_{\Lambda}$ in terms of this auxiliary function.


\begin{proposition}\label{lequiv20}
Given $\Lambda=\{\lambda_1,\lambda_2,\ldots,\lambda_j,\ldots\}$ a set of exponents which has an integral basis, let $f_1$ and $f_2$ be two almost periodic functions in the class $\mathcal{D}_{\Lambda}$ whose Dirichlet series are given by $\sum_{j\geq 1}a_je^{\lambda_js}$ and $\sum_{j\geq 1}b_je^{\lambda_js}$ respectively. Let $\mathbf{g}:=(g_1,g_2,\ldots,g_k,\ldots)$ be the vector of the elements of an integral basis $G_{\Lambda}$ for $\Lambda$. Thus $f_1$ is equivalent to $f_2$ if and only if there exists some $\mathbf{x}\in \mathbb{R}^{\sharp G_\Lambda}$ such that
$$\sum_{j\geq 1}b_je^{\lambda_j(\sigma+it)}=F_{f_1}(\sigma,\mathbf{x}+t\mathbf{g})$$ for $\sigma+it\in U$, where $U$ is an open vertical strip so that $f_2\in AP(U,\mathbb{C})$.
\end{proposition}
\begin{proof}
Let $\sum_{j\geq 1}a_je^{\lambda_js}$ and $\sum_{j\geq 1}b_je^{\lambda_js}$ be the Dirichlet series associated with $f_1$ and $f_2$ respectively. Let $U$ be an open vertical strip so that $f_2\in AP(U,\mathbb{C})$.
If $f_1\sim f_2$, then \cite[Proposition 1]{SV} assures the existence of
$\mathbf{x}_0\in \mathbb{R}^{\sharp\Lambda}$
such that $b_j=a_j e^{<\mathbf{r}_j,\mathbf{x}_0>i}$ for $j\geq 1$. Thus, fixed $s=\sigma+it\in U$, we have
$$\displaystyle{\sum_{j\geq 1}b_je^{\lambda_j(\sigma+it)}=\sum_{j\geq1}a_je^{i<\mathbf{r}_j,\mathbf{x}_0>}e^{\lambda_j\sigma}e^{i\lambda_j t}}= \sum_{j\geq1}a_je^{\lambda_j\sigma}e^{i<\mathbf{r}_j,\mathbf{x}_0>}e^{it<\mathbf{r}_j,\mathbf{g}>}=$$
$$F_{f_1}(\sigma,\mathbf{x}_0+t\mathbf{g}).$$
Conversely, suppose the existence of $\mathbf{x}_0\in \mathbb{R}^{\sharp\Lambda}$
such that
$\sum_{j\geq 1}b_je^{\lambda_j(\sigma+it)}=F_{f_1}(\sigma,\mathbf{x}_0+t\mathbf{g})$ for any $\sigma+it\in U$. Hence $$\sum_{j\geq 1}b_je^{\lambda_j(\sigma+it)}=\sum_{j\geq1}a_je^{i<\mathbf{r}_j,\mathbf{x}_0>}e^{\lambda_j (\sigma+it)}\ \forall \sigma+it\in U.$$
Now, by the uniqueness of the coefficients of an exponential sum in $\mathcal{D}_{\Lambda}$, it is clear that
$b_j=a_j e^{<\mathbf{r}_j,\mathbf{x}_0>i}$ for each $j\geq 1$, which shows that $f_1\sim f_2$.
\end{proof}

We next define the following set which will be widely used from now on.

\begin{definition}\label{image}
Given $\Lambda=\{\lambda_1,\lambda_2,\ldots,\lambda_j,\ldots\}$ a set of exponents which has an integral basis, let $f(s)\in \mathcal{D}_{\Lambda}$ be an almost periodic function in an open vertical strip $U$, and $\sigma_0=\operatorname{Re}s_0$ with $s_0\in U$. We define $\operatorname{Img}\left(F_f(\sigma_0,\mathbf{x})\right)$ to be the set of values in the complex plane taken on by the auxiliary function $F_f(\sigma,\mathbf{x})$, relative to the integral basis $G_{\Lambda}$, when $\sigma=\sigma_0$; that is
$$\operatorname{Img}\left(F_f(\sigma_0,\mathbf{x})\right)=\{s\in\mathbb{C}:\exists \mathbf{x}\in\mathbb{R}^{\sharp G_{\Lambda}}\mbox{ such that }s=F_f(\sigma_0,\mathbf{x})\}.$$
\end{definition}

The notation $\operatorname{Img}\left(F_f(\sigma_0,\mathbf{x})\right)$ is well-posed because this set is independent of the integral basis $G_{\Lambda}$ such as the following lemma shows.

\begin{lemma}\label{indep}
Given $\Lambda$ a set of exponents and $G_{\Lambda}$ an integral basis for $\Lambda$, let $f(s)\in \mathcal{D}_{\Lambda}$ be an almost periodic function in an open vertical strip $U$, and $\sigma_0=\operatorname{Re}s_0$ with $s_0\in U$. Then the set $\operatorname{Img}\left(F_f(\sigma_0,\mathbf{x})\right)$ is independent of $G_{\Lambda}$.
\end{lemma}
\begin{proof}
Let $\sum_{j\geq 1}a_je^{\lambda_js}$ be the Dirichlet series associated with $f(s)\in \mathcal{D}_{\Lambda}$, and  $G_{\Lambda}$ and $H_{\Lambda}$ be two integral basis for $\Lambda$. Also, let $\operatorname{Img}\left(F_f^{G_{\Lambda}}(\sigma_0,\mathbf{x})\right)$ and $\operatorname{Img}\left(F_f^{H_{\Lambda}}(\sigma_0,\mathbf{x})\right)$ be the set of values in the complex plane taken on by the auxiliary functions, relative to the basis $G_{\Lambda}$ and $H_{\Lambda}$ respectively.
For each $j\geq 1$ let $\mathbf{r}_j$ and $\mathbf{s}_j$ be the vector of integer components so that  $\lambda_j=<\mathbf{r}_j,\mathbf{g}>$ and $\lambda_j=<\mathbf{s}_j,\mathbf{h}>$, with $\mathbf{g}$ and $\mathbf{h}$ the vectors associated with the basis $G_{\Lambda}$ and $H_{\Lambda}$, respectively. Finally, for each $k\geq 1$, let $\mathbf{t}_k$ be the vector so that $h_k=<\mathbf{t}_k,\mathbf{g}>$.
Take $w_1\in \operatorname{Img}\left(F_f^{G_{\Lambda}}(\sigma_0,\mathbf{x})\right)$, then there exists $\mathbf{x}_1\in\mathbb{R}^{\sharp G_{\Lambda}}\mbox{ such that }w_1=F_f^{G_{\Lambda}}(\sigma_0,\mathbf{x}_1)$. Hence
$$w_1=F_f^{G_{\Lambda}}(\sigma_0,\mathbf{x}_1)=\sum_{j\geq1}a_j e^{\lambda_j\sigma_0
}e^{<\mathbf{r}_j,\mathbf{x}_1>i}=\sum_{j\geq1}a_j e^{\lambda_j\sigma_0
}e^{<\mathbf{s}_j,\mathbf{x}_2>i},$$
where $\mathbf{x}_2$ is defined as $x_{2,k}=<\mathbf{t}_k,\mathbf{x}_1>$ for each $k\geq 1$. Therefore, $w_1=F_f^{H_{\Lambda}}(\sigma_0,\mathbf{x}_2)$ and $w_1\in \operatorname{Img}\left(F_f^{H_{\Lambda}}(\sigma_0,\mathbf{x})\right)$, which gives $$\operatorname{Img}\left(F_f^{G_{\Lambda}}(\sigma_0,\mathbf{x})\right)\subseteq \operatorname{Img}\left(F_f^{H_{\Lambda}}(\sigma_0,\mathbf{x})\right).$$ An analogous argument shows that
$\operatorname{Img}\left(F_f^{H_{\Lambda}}(\sigma_0,\mathbf{x})\right)\subseteq \operatorname{Img}\left(F_f^{G_{\Lambda}}(\sigma_0,\mathbf{x})\right)$, which proves the result.
\end{proof}

The following lemma will be used in order to obtain some important results of this paper.

\begin{lemma}\label{lemaulti}
Given $\Lambda$ a set of exponents which has an integral basis, let $f(s)\in \mathcal{D}_{\Lambda}$ be an almost periodic function in a vertical strip $\{s=\sigma+it:\alpha<\sigma<\beta\}$. Consider $E$ a compact set of real numbers included in $(\alpha,\beta)$. 
Thus
$\bigcup_{\sigma\in E}\operatorname{Img}\left(F_f(\sigma,\mathbf{x})\right)$ is closed.
\end{lemma}
\begin{proof}
Let $w_1,w_2,\ldots,w_j,\ldots$ be a sequence tending to $w_0$, where $$w_j\in\bigcup_{\sigma\in E}\operatorname{Img}\left(F_f(\sigma,\mathbf{x})\right)\mbox{ for each }j\in\mathbb{N}.$$ Thus, to each $w_n$ there corresponds a $\sigma_n\in E$ such that $w_n\in \operatorname{Img}\left(F_f(\sigma_n,\mathbf{x})\right)$ or equivalently $w_n=F_f(\sigma_n,\mathbf{x}_n)$ for some vector $\mathbf{x}_n$. Since $E$ is compact, there exists a subsequence $\{\sigma_{n_k}\}_k\subset \{\sigma_n\}_n$ which converges to $\sigma_0\in E$.
If $\sum_{j\geq 1}a_je^{\lambda_js}$ is the Dirichlet series associated with $f(s)$, then we can write
\begin{equation}\label{newq}
w_{n_k}=\sum_{j\geq 1}a_je^{\lambda_j\sigma_{n_k}}e^{i<\mathbf{r}_j,\mathbf{x}_{n_k}>}=\sum_{j\geq 1}a_je^{\lambda_j\sigma_{n_k}}\prod_{m=1}^{q_j}s_{n_k,m}^{r_{j,m}},
\end{equation}
where the values $s_{n_k,m}:=e^{ix_{n_k,m}}$ are on the unit circle and $\mathbf{r}_j$ is the vector of entire components associated with an integral basis for $\Lambda$. Then the sequence $\{s_{n_1,1},s_{n_2,1},\ldots,s_{n_k,1},\ldots\}$ contains a subsequence $\{s_{l_m,1}\}_m$ which converges to a point $s_{0,1}=e^{ix_{0,1}}$, $x_{0,1}\in[0,2\pi)$, on the unit circle. Thus we consider the sequence given by $\{s_{l_1,2},s_{l_2,2},\ldots,s_{l_m,2},\ldots\}$ and we notice that it contains a subsequence $\{s_{q_m,2}\}_m$ which converges to a point $s_{0,2}=e^{ix_{0,2}}$, $x_{0,2}\in[0,2\pi)$, on the unit circle. Nextly, we move then to $\{s_{q_1,3},w_{q_2,3},\ldots,w_{q_k,3},\ldots\}$ and so on. In this way, we construct a vector $\mathbf{s}_0=(s_{0,1},s_{0,2},\ldots)$ on the unit circle which, by taking the limit in (\ref{newq}), satisfies
$$w_0=\sum_{j\geq 1}a_je^{\lambda_j\sigma_{0}}\prod_{m=1}^{q_j}s_{0,m}^{r_{j,m}}=\sum_{j\geq 1}a_je^{\lambda_j\sigma_{0}}e^{i<\mathbf{r}_j,\mathbf{x}_0>}=F_f(\sigma_0,\mathbf{x}_0),$$
where $\mathbf{x}_0:=(x_{0,1},x_{0,2},\ldots)$.
Hence $w_0\in \operatorname{Img}\left(F_f(\sigma_0,\mathbf{x})\right)$ and the result follows.
\end{proof}

\section{Main results}

Given a function $f(s)$, take the notation $$\operatorname{Img}\left(f(\sigma_0+it)\right)=\{s\in\mathbb{C}:\exists t\in\mathbb{R}\mbox{ such that }s=f(\sigma_0+it)\}.$$
\begin{remark}\label{remark4}
Consider $f\in AP(U,\mathbb{C})$, for some vertical strip $U=\{s\in\mathbb{C}:\alpha<\operatorname{Re}s<\beta\}$, whose Dirichlet series is given by $\sum_{j\geq 1}a_je^{\lambda_js}$. Note that any $w_0\in\operatorname{Img}\left(f(\sigma_0+it)\right)$, with $\sigma_0\in (\alpha,\beta)$, can be obtained as uniform limit of the Bochner-Fej\'{e}r polynomials which converge to $f(s)$ on every reduced strip in $U$ and formally to its associated Dirichlet series. Hence $w_0$ can also be written as $\sum_{j\geq 1}a_je^{\lambda_j(\sigma_0+it_0)}$ for some $t_0\in\mathbb{R}$, and vice versa. In fact, we recall that if the Dirichlet series is uniformly convergent on a vertical strip $U$, then it coincides with $f(s)$.
\end{remark}

We next show the first important result in this paper concerning the connection between our equivalence relation and the set of values in the complex
plane taken on by the auxiliary function.

\begin{proposition}\label{pult}
Given $\Lambda$ a set of exponents which has an integral basis, let $f(s)\in \mathcal{D}_{\Lambda}$ be an almost periodic function in an open vertical strip $U$, and $\sigma_0=\operatorname{Re}s_0$ with $s_0\in U$.
\begin{itemize}
\item[i)] If $f_1\sim f$, then $\operatorname{Img}\left(f_1(\sigma_0+it)\right)\subset \overline{\operatorname{Img}\left(f(\sigma_0+it)\right)}$ and $$\overline{\operatorname{Img}\left(f(\sigma_0+it)\right)}= \overline{\operatorname{Img}\left(f_1(\sigma_0+it)\right)}.$$

\item[ii)] $\operatorname{Img}\left(F_f(\sigma_0,\mathbf{x})\right)=\bigcup_{f_k\sim f}\operatorname{Img}\left(f_k(\sigma_0+it)\right).$


\end{itemize}
\end{proposition}
\begin{proof}

i) Under the assumption of the existence of an integral basis for $\Lambda$, \cite[Theorem 4]{SV} shows that the functions in the same equivalence class are obtained as limit points of $\mathcal{T}_f=\{f_{\tau}(s):=f(s+i\tau):\tau\in\mathbb{R}\}$, that is, any function $f_1\sim f$ is the limit (in the sense of the uniform convergence on every reduced strip of $U$) of a sequence $\{f_{\tau_n}(s)\}$ with $f_{\tau_n}(s):=f(s+i\tau_n)$.
    Take $w_1\in \operatorname{Img}\left(f_1(\sigma_0+it)\right)$, then there exists $t_1\in\mathbb{R}$ such that $w_1=f_1(\sigma_0+it_1)$. Now, given $\varepsilon>0$ there exists $\tau>0$ such that $|f_1(\sigma_0+it_1)-f_{\tau}(\sigma_0+it_1)|<\varepsilon$, which means that
$$ |w_1-f(\sigma_0+i(t_1+\tau))|<\varepsilon.$$
Now it is immediate that $w_1\in \overline{\operatorname{Img}\left(f(\sigma_0+it)\right)}$ and consequently $$\operatorname{Img}\left(f_1(\sigma_0+it)\right)\subset \overline{\operatorname{Img}\left(f(\sigma_0+it)\right)}.$$
Analogously, by symmetry we have $\operatorname{Img}\left(f(\sigma_0+it)\right)\subset \overline{\operatorname{Img}\left(f_1(\sigma_0+it)\right)}$, which implies that     $$\overline{\operatorname{Img}\left(f(\sigma_0+it)\right)}= \overline{\operatorname{Img}\left(f_1(\sigma_0+it)\right)}.$$

ii)
Take $w_0\in \bigcup_{f_k\sim f}\operatorname{Img}\left(f_k(\sigma_0+it)\right)$, then $w_0\in \operatorname{Img}\left(f_k(\sigma_0+it)\right)$ for some $f_k\sim f$, which means that there exists $t_0\in\mathbb{R}$ such that $$w_0=f_k(\sigma_0+it_0).$$ Note that, by taking Remark \ref{remark4} into account, Proposition \ref{lequiv20} assures the existence of a vector $\mathbf{x}_0$ such that $w_0=F_{f}(\sigma_0,\mathbf{x}_0+t_0\mathbf{g})$. Hence $w_0=F_{f}(\sigma_0,\mathbf{y}_0)$, with $\mathbf{y}_0=\mathbf{x}_0+t_0\mathbf{g}$, which means that $w_0\in \operatorname{Img}\left(F_f(\sigma_0,\mathbf{x})\right)$. Conversely, if $w_0\in \operatorname{Img}\left(F_f(\sigma_0,\mathbf{x})\right)$, then
$w_0=F_{f}(\sigma_0,\mathbf{y}_0)$ for some $\mathbf{y}_0\in\mathbb{R}^{\sharp G_{\Lambda}}$. Take $t_0\in\mathbb{R}$. Since $\mathbf{y}_0=\mathbf{x}_0+t_0\mathbf{g}$, with $\mathbf{x}_0=\mathbf{y}_0-t_0\mathbf{g}$, then
$$w_0=F_{f}(\sigma_0,\mathbf{x}_0+t_0\mathbf{g})=\sum_{j\geq1}a_j e^{\lambda_j\sigma_0
}e^{<\mathbf{r}_j,\mathbf{x}_0+t_0\mathbf{g}>i}=\sum_{j\geq1}a_j e^{\lambda_j(\sigma_0+it_0)
}e^{<\mathbf{r}_j,\mathbf{x}_0>i}.$$
Hence $\sum_{j\geq1}a_je^{<\mathbf{r}_j,\mathbf{x}_0>i} e^{\lambda_js}$ is the associated Dirichlet series of an almost periodic function $h(s)\in AP(U,\mathbb{C})$ such that $h\sim f$ (see \cite[Lemma 3]{SV}) and, by taking Remark \ref{remark4} into account, we have that $w_0=h(\sigma_0+it_0)$, which shows that $w_0\in \bigcup_{f_k\sim f}\operatorname{Img}\left(f_k(\sigma_0+it)\right).$
    \end{proof}

We next prove another important equality. 

\begin{proposition}\label{pultnew}
Given $\Lambda$ a set of exponents which has an integral basis, let $f(s)\in \mathcal{D}_{\Lambda}$ be an almost periodic function in an open vertical strip $U$, and $\sigma_0=\operatorname{Re}s_0$ with $s_0\in U$. Then
    $\operatorname{Img}\left(F_f(\sigma_0,\mathbf{x})\right)=\overline{\operatorname{Img}\left(f_1(\sigma_0+it)\right)}$ for any $f_1\sim f$.
\end{proposition}
\begin{proof}
Let $\mathbf{g}$ be the vector associated with a basis $G_{\Lambda}$. Since the Fourier series of $f_{\sigma_0}(t):=f(\sigma_0+it)$ can be obtained as $F_f(\sigma_0,t\mathbf{g})$, with $t\in\mathbb{R}$, by taking Remark \ref{remark4} into account, it is clear that $\operatorname{Img}\left(f(\sigma_0+it)\right)\subset \operatorname{Img}\left(F_f(\sigma_0,\mathbf{x})\right)$. On the other hand,  we deduce from i) and ii) of Proposition \ref{pult} that
    $$\operatorname{Img}\left(f(\sigma_0+it)\right)\subset \operatorname{Img}\left(F_f(\sigma_0,\mathbf{x})\right)=\bigcup_{f_k\sim f}\operatorname{Img}\left(f_k(\sigma_0+it)\right)\subset \overline{\operatorname{Img}\left(f(\sigma_0+it)\right)}.$$
    Finally, by taking the closure and from Lemma \ref{lemaulti}, we conclude that
    $$\operatorname{Img}\left(F_f(\sigma_0,\mathbf{x})\right)=\overline{\operatorname{Img}\left(f(\sigma_0+it)\right)}.$$
Now, the result follows from property i) of Proposition \ref{pult}.
\end{proof}

If $E$ is an arbitrary set of real numbers included in the real projection of the vertical strip $U$ of almost periodicity of a function $f\in AP(U,\mathbb{C})$, we next study the set $\bigcup_{\sigma\in E}\operatorname{Img}\left(F_f(\sigma,\mathbf{x})\right)$.

\begin{proposition}\label{pult2}
Given $\Lambda$ a set of exponents which has an integral basis, let $f(s)\in \mathcal{D}_{\Lambda}$ be an almost periodic function in a vertical strip $\{s=\sigma+it:\alpha<\sigma<\beta\}$. Consider $E$ an arbitrary set of real numbers included in $(\alpha,\beta)$. 
Thus
$$\bigcup_{\sigma\in E}\operatorname{Img}\left(f(\sigma+it)\right)\subset \bigcup_{\sigma\in E}\operatorname{Img}\left(F_f(\sigma,\mathbf{x})\right)\subset \overline{\bigcup_{\sigma\in E}\operatorname{Img}\left(f(\sigma+it)\right)}.$$
\end{proposition}
\begin{proof}
If $w_0\in\bigcup_{\sigma\in E}\operatorname{Img}\left(f(\sigma+it)\right)$ then $w_0\in \operatorname{Img}\left(f(\sigma_0+it)\right)$ for some $\sigma_0\in E$. Now, by Proposition \ref{pult}, we have $$\operatorname{Img}\left(f(\sigma_0+it)\right)\subset \operatorname{Img}\left(F_f(\sigma_0,\mathbf{x})\right)\subset \bigcup_{\sigma\in E}\operatorname{Img}\left(F_f(\sigma,\mathbf{x})\right).$$
Moreover, if $w_0\in \bigcup_{\sigma\in E}\operatorname{Img}\left(F_f(\sigma,\mathbf{x})\right)$ then $w_0\in\operatorname{Img}\left(F_f(\sigma_0,\mathbf{x})\right)$ for some $\sigma_0\in E$ and, by Proposition \ref{pultnew}, $$w_0\in \overline{\operatorname{Img}\left(f(\sigma_0+it)\right)}.$$ Finally, it is clear that $\overline{\operatorname{Img}\left(f(\sigma_0+it)\right)}\subset \overline{\bigcup_{\sigma\in E}\operatorname{Img}\left(f(\sigma+it)\right)}$ and hence the result holds.
\end{proof}

As a consequence of the result above, we formulate the following corollary for the case that the set $E$ is compact.

\begin{corollary}\label{pultt}
Given $\Lambda$ a set of exponents which has an integral basis, let $f(s)\in \mathcal{D}_{\Lambda}$ be an almost periodic function in a vertical strip $\{\sigma+it\in\mathbb{C}:\alpha<\sigma<\beta\}$. Consider $E$ a compact set of real numbers included in $(\alpha,\beta)$. 
Thus
$$\bigcup_{\sigma\in E}\operatorname{Img}\left(F_f(\sigma,\mathbf{x})\right)= \overline{\bigcup_{\sigma\in E}\operatorname{Img}\left(f(\sigma+it)\right)}=\overline{\bigcup_{\sigma\in E}\operatorname{Img}\left(f_1(\sigma+it)\right)},$$
for any $f_1\sim f$.
\end{corollary}
\begin{proof}
It is clear from Proposition \ref{pult2}, Lemma \ref{lemaulti} and part i) of Proposition \ref{pult}.
\end{proof}

As the following example shows, the converse of Corollary \ref{pultt}, and in particular the converse of part i) in Proposition \ref{pult}, is not true.
\begin{example}\label{ex1nr}
Given $\Lambda=\{\log 2,\log 3,\log 5\}$, consider $f_1(s)=e^{s\log 2}+e^{s\log 3}+2e^{s\log 5}$ and $f_2(s)=e^{s\log 2}+2e^{s\log 3}+e^{s\log 5}$, which are two exponential polynomials in $\mathcal{P}_{\Lambda}$. The auxiliary functions associated with $f_1$ and $f_2$ are
$$F_{f_1}(\sigma,\mathbf{x})=2^{\sigma}e^{x_1 i}+3^{\sigma}e^{x_2 i}+2\cdot 5^{\sigma}e^{x_3 i}$$
and
$$F_{f_2}(\sigma,\mathbf{x})=2^{\sigma}e^{x_1 i}+2\cdot 3^{\sigma}e^{x_2 i}+5^{\sigma}e^{x_3 i},$$
respectively, with $\mathbf{x}=(x_1,x_2,x_3)$ (see Definition \ref{auxuliaryfunc}). Take $\sigma_0=0$. In this case, since $F_{f_1}(0,x_1,x_2,x_3)=F_{f_2}(0,x_1,x_3,x_2)$, it is clear that $$\operatorname{Img}\left(F_{f_1}(\sigma_0,\mathbf{x})\right)=\operatorname{Img}\left(F_{f_2}(\sigma_0,\mathbf{x})\right).$$ Therefore, by part iii) of Proposition \ref{pult} (or Corollary \ref{pultt}), we have $$\overline{\operatorname{Img}\left(f_1(\sigma_0+it)\right)}=\overline{\operatorname{Img}\left(f_2(\sigma_0+it)\right)}.$$
However, it is immediate by Definition \ref{DefEquiv} that $f_1$ and $f_2$ are not equivalent.
\end{example}

At this point we will demonstrate an extension of Bohr's equivalence theorem \cite[Section 8.11]{Apostol}. Given $\Lambda$ a set of exponents for which there exists an integral basis, let $f_1,f_2\in \mathcal{D}_{\Lambda}$ be two equivalent almost periodic functions. We next show that, in any open half-plane or open vertical strip included in their region of almost periodicity, the functions $f_1$ and $f_2$ take the same set of values.

\begin{theorem}\label{beqg}
Given $\Lambda$ a set of exponents which has an integral basis, let $f_1,f_2\in \mathcal{D}_{\Lambda}$ be two equivalent almost periodic functions in a vertical strip $\{\sigma+it\in\mathbb{C}:\alpha<\sigma<\beta\}$. Consider $E$ an open set of real numbers included in $(\alpha,\beta)$.
Thus $$\bigcup_{\sigma\in E}\operatorname{Img}\left(f_1(\sigma+it)\right)=\bigcup_{\sigma\in E}\operatorname{Img}\left(f_2(\sigma+it)\right).$$
That is, the functions $f_1$ and $f_2$ take the same set of values on the region $\{s=\sigma+it\in\mathbb{C}:\sigma\in E\}$.
\end{theorem}
\begin{proof} 
Without loss of generality, suppose that $f_1$ and $f_2$ are not constant functions (otherwise it is trivial). Take $w_0\in \bigcup_{\sigma\in E}\operatorname{Img}\left(f_1(\sigma+it)\right)$, then $w_0\in \operatorname{Img}\left(f_1(\sigma_0+it)\right)$ for some $\sigma_0\in E$ and hence $w_0=f_1(\sigma_0+it_0)$ for some $t_0\in\mathbb{R}$. Furthermore, by Proposition \ref{pult}, we get
$w_0\in\overline{\operatorname{Img}\left(f_1(\sigma_0+it)\right)}=\overline{\operatorname{Img}\left(f_2(\sigma_0+it)\right)}$,
which implies that there exists a sequence $\{t_n\}$ of real numbers such that
$$w_0=\lim_{n\to\infty}f_2(\sigma_0+it_n).$$ Take $h_n(s):=f_2(s+it_n)$, $n\in\mathbb{N}$.
By \cite[Proposition 4]{SV}, there exists a subsequence $\{h_{n_k}\}_k\subset \{h_n\}_n$ which converges uniformly on compact subsets to a function $h(s)$, with $h\sim f_2$. Observe that $$\lim_{k\to\infty}h_{n_k}(\sigma_0)=h(\sigma_0)=w_0.$$ Therefore, by Hurwitz's theorem \cite[Section 5.1.3]{Ash2}, there is a positive integer $k_0$ 
such that for $k>k_0$ the functions $h^*_{n_k}(s):=h_{n_k}(s)-w_0$ have one zero in $D(\sigma_0,\varepsilon)$ for any $\varepsilon>0$ sufficiently small. This means that for $k>k_0$ the functions $h_{n_k}(s)=f_2(s+it_{n_k})$, and hence the function $f_2(s)$, take the value $w_0$ on the region $\{s=\sigma+it:\sigma_0-\varepsilon<\sigma<\sigma_0+\varepsilon\}$ for any $\varepsilon>0$ sufficiently small (recall that $E$ is an open set). 
Consequently, $w_0\in \bigcup_{\sigma\in E}\operatorname{Img}\left(f_2(\sigma+it)\right)$.
We analogously prove that $\bigcup_{\sigma\in E}\operatorname{Img}\left(f_2(\sigma+it)\right)\subset \bigcup_{\sigma\in E}\operatorname{Img}\left(f_1(\sigma+it)\right)$.
\end{proof}

As the following example shows, fixed an open set $E$ in $(\alpha,\beta)$, the converse of Theorem \ref{beqg} is not true.
\begin{example}\label{exi}
Given $\Lambda=\{\log 2,\log 3,\log 5\}$, consider the functions of Exam\-ple \ref{ex1nr}, $f_1(s)=e^{s\log 2}+e^{s\log 3}+2e^{s\log 5}$ and $f_2(s)=e^{s\log 2}+2e^{s\log 3}+e^{s\log 5}$, which are two non-equivalent exponential polynomials in $\mathcal{P}_{\Lambda}$. Take $E=(-\infty,0)$. We next demonstrate that
$$\bigcup_{\sigma\in E}\operatorname{Img}\left(f_1(\sigma+it)\right)=\bigcup_{\sigma\in E}\operatorname{Img}\left(f_2(\sigma+it)\right)=\{s\in\mathbb{C}:|s|<4\}.$$
Indeed, if $w\in \bigcup_{\sigma\in E}\operatorname{Img}\left(f_1(\sigma+it)\right)$ then $w=f_1(\sigma_0+it_0)$ for some $\sigma_0<0$ and $t_0\in\mathbb{R}$. Hence $$|w|=|2^{\sigma_0}e^{it_0\log 2}+3^{\sigma_0}e^{it_0\log 3}+2\cdot 5^{\sigma_0}e^{it_0\log 5}|\leq 2^{\sigma_0}+3^{\sigma_0}+2\cdot 5^{\sigma_0}<4.$$ Take now $w_0\in D(0,4)$, then $w_0=re^{i\theta}$ for some $0\leq r<4$ and $\theta\in\mathbb{R}$. Note that the auxiliary function associated with $f_1$, which is
$F_{f_1}(\sigma,\mathbf{x})=2^{\sigma}e^{x_1 i}+3^{\sigma}e^{x_2 i}+2\cdot 5^{\sigma}e^{x_3 i},$
verifies
$$F_{f_1}(\sigma,\theta,\theta,\theta)=(2^{\sigma}+3^{\sigma}+2\cdot 5^{\sigma})e^{i\theta}.$$
Hence there exists $\sigma_1<0$ such that $F_{f_1}(\sigma_1,\theta,\theta,\theta)=w_0$ and thus, by Proposition \ref{pult}, we have $$w_0\in \overline{\operatorname{Img}(f_1(\sigma_1+it))}.$$
Now, since $E$ is an open set, by following the proof of Theorem \ref{beqg}, we get $$w_0\in \bigcup_{\sigma\in E}\operatorname{Img}\left(f_1(\sigma+it)\right).$$
We can analogously prove that $\bigcup_{\sigma\in E}\operatorname{Img}\left(f_2(\sigma+it)\right)=D(0,4)$.
\end{example}

\bibliographystyle{amsplain}

\begin{thebibliography}{9}
\bibitem{Apostol} T.M. Apostol, Modular functions and Dirichlet series in number theory, Springer-Verlag, New York, 1990.

\bibitem{Ash2} \textsc{Ash, R.B.; Novinger, W.P.}: \textit{Complex Variables}, New York: Academic Press,
2004. 

\bibitem{Avellar} C.E. Avellar, J.K. Hale, On the zeros of exponential
polynomials. J. Math. Anal. Appl., \textbf{73} (1980), 434-452.

\bibitem{Besi} A.S. Besicovitch, Almost periodic functions, Dover, New York, 1954.


\bibitem{BohrDirichlet} H. Bohr, Z\"{u}r Theorie der allgemeinen Dirichletschen Reihen, Math. Ann., \textbf{79} (1919), 136-156.


\bibitem{Bohr} H. Bohr, Almost periodic functions, Chelsea, New York, 1951.

\bibitem{Bohr2} H. Bohr, Contribution to the theory of almost periodic functions, Det Kgl. danske Videnskabernes Selskab. Matematisk-fisiske meddelelser. Bd. XX. Nr. 18, Copenhague, 1943.


\bibitem{Corduneanu1} C. Corduneanu, Almost Periodic Functions, Interscience publishers, New York, London, Sydney, Toronto, 1968.

\bibitem{Corduneanu} C. Corduneanu, Almost Periodic Oscillations and Waves, Springer, New York, 2009.




\bibitem{Jessen} B. Jessen, Some aspects of the theory of almost periodic functions, in Proc. Internat. Congress Mathematicians Amsterdam, 1954, Vol. 1, North-Holland, 1954, pp. 304--351.




\bibitem{Rigue} M. Riguetti, On Bohr's equivalence theorem. J. Math. Anal. Appl., \textbf{445} (1) (2017), 650-654. corrigendum ibid. 449 (2017), 939-940. 

\bibitem{nonisolation} J.M. Sepulcre, T. Vidal, On the non-isolation of the real projections of the zeros of exponential polynomials, J. Math. Anal. Appl., \textbf{437} (2016), 513-525.

\bibitem{SV} J.M. Sepulcre, T. Vidal, Almost Periodic Functions in terms of Bohr's Equivalence Relation, Ramanujan J., DOI: 10.1007/s11139-017-9950-1, 2017. Corrigendum, to appear.

\bibitem{Spira} R. Spira, Sets of values of general Dirichlet series, Duke Math. J. \textbf{35} (1) (1968), 79-82.

\bibitem{Spira2} R. Spira, Zeros of Sections of the Zeta Function II., Math. Comp., \textbf{22} (101) (1968), 163-173.


\end{thebibliography}

\end{document}